\documentclass{amsart}

\usepackage{latexsym}
\usepackage{amsmath}
\usepackage{amssymb}

\usepackage{multicol}

\usepackage[english]{babel}

\theoremstyle{plain}

\newtheorem{theorem}{Theorem}[section]
\newtheorem{proposition}{Proposition}[section]
\newtheorem{lemma}{Lemma}[section]

\theoremstyle{definition}

\newcommand{\za}{\zeta}

\newcommand{\ph}{\varphi}
\newcommand{\cph}{C_\varphi}
\newcommand{\hol}{\mathcal{H}ol}
\newcommand{\Dbb}{\mathbb D}
\newcommand{\Tbb}{\mathbb T}

\newcommand{\Cbb}{\mathbb C}

\newcommand{\bloch}{\mathcal{B}}

\DeclareDocumentCommand \lclass {m o} {\ensuremath{\mathrm L_{#1} \IfValueT {#2} {\left( #2 \right)} }}
 \DeclareDocumentCommand \hclass {m o} {\ensuremath{\mathrm H_{#1} \IfValueT {#2} {\left( #2 \right)} }}
 \DeclareDocumentCommand \hrclass {m o} {\ensuremath{\mathcal H_{#1} \IfValueT {#2} {\left( #2 \right)} }}
\DeclareDocumentCommand \apclass {m} {\ensuremath{\mathrm A_{#1}  }}
\newcommand {\nplus} {\ensuremath {\mathrm {N}^+}}

\numberwithin{equation}{section}

\begin{document}

\date{}

\author{Evgeny Abakumov}
\address{Univ Gustave Eiffel,
Univ Paris Est Creteil, CNRS, LAMA UMR8050 F-77447 Marne-la-Vall\'ee, France}
\email{evgueni.abakoumov@univ-eiffel.fr}

\author{Evgueni Doubtsov}
\address{St.~Petersburg Department
of Steklov Mathematical Institute, Fo\-ntanka 27, St.~Petersburg 191023, Russia}
\email{dubtsov@pdmi.ras.ru}

\author{Dmitry Rutsky}
\address{St.~Petersburg Department
of Steklov Mathematical Institute, Fo\-ntanka 27, St.~Petersburg 191023, Russia}
\email{rutsky@pdmi.ras.ru}

\title[Square function characterization]{Square function characterization \\ of Hardy-type spaces}

\begin{abstract}
The well-known characterization of Hardy spaces~$\mathrm{H}_{p}(\mathbb D)$, $0 < p < \infty$,
  in terms of the Littlewood-Paley
  g-function
  $$
  S f (\zeta) = \left(\int_0^1 |f' (r \zeta)|^2 (1 - r) \mathrm dr\right)^{1/2} \in \mathrm{L}_{p}
  $$
  is generalized to Hardy-type spaces~$X_A$ corresponding to quasi-Banach lattices~$X$
  on the unit circle~$\mathbb T$ under the assumption that the Hardy-Littlewood maximal operator~$M$
  is bounded in~$(X^\delta)'$ with some~$\delta > 0$.
  As an application to composition operators~$C_\varphi$, we derive an exact criterion for
  the boundedness and compactness 
  of~$C_\varphi : \bloch^\omega \to X_A$, where~$\bloch^\omega = \{f \mid \sup |f'|/\omega < \infty\}$
  is the weighted Bloch space with a log-convex radial weight~$\omega$,
  generalizing recent results in the one-dimensional setting.
\end{abstract}



\maketitle

\section {Introduction}

Various generalizations of the classical Hardy spaces have been studied over the years.
In the present paper, our goal is to extend the classical characterization
in terms of the Littlewood-Paley $g$-function~$S$ (see, e.g., \cite {pavlovic2013} and references contained therein),
that for brevity  we will simply call the square function,
to the Hardy-type spaces corresponding
to arbitrary quasi-Banach lattices of measurable functions satisfying a natural assumption
that by now seems to be standard for various results concerning
the real harmonic analysis on Banach Function Spaces.
This is done in the following Section~\ref {s:hts}; see Theorem~\ref {t:htssfc}.
We also give a very short proof of the extrapolation theorem from \cite {weightsandextrapolation} that we use;
see Theorem~\ref {t:extr}.

Our main motivation is an application to the composition operators given in Section~\ref {s:comp}
below; see Theorem~\ref {t:cph}.
In the one-dimensional case,
this generalizes a recent result by Chen and Hamada~\cite {ChH1}, \cite {ChH2} for the classical Hardy spaces
to a wide class of Hardy-type spaces.
However, this characterization also seems to be of independent interest:
despite the huge progress being made in the relevant areas,
and the perception that every aspect of this result is by now seemingly well-understood,
it is difficult to find in the literature any references that we could readily use.

Restricting ourselves to the one-dimensional analytic spaces on the unit disc
also lends itself to a rather
succinct and accessible treatment, although we still have to rely on a number of relatively well-known
but fairly advanced modern tools to achieve the stated generality.
It seems that very similar extensions are feasible for the Hardy spaces on strictly convex domains,
with the standard tools of real harmonic analysis on spaces of homogeneous type replacing
the arguments based on factorization, but this is, perhaps, better done in a separate publication.

\section {Hardy-type spaces}
\label {s:hts}

A convenient framework for
Hardy-type spaces corresponding to general lattices of measurable functions 
was introduced by S.~V.~Kislyakov
in~\cite {Kisliakov1999InterpolationOfHpSpacesSomeRecentDevelopments} and then substantially
refined in~\cite {kisliakov2002en}; here we adopt the latter,
briefly going over the particular features that we need.

Let $\Dbb$ denote the open unit disc of $\Cbb$ and let $\Tbb$
denote the unit circle.
Let $\hol(\Dbb)$ denote the space of holomorphic functions on~$\Dbb$.

A quasi-normed lattice of mea\-surable functions~$X$ on~$\mathbb T$ is a quasi-normed space of
measurable functions~$X$ such that the quasi-norm is compatible with the natural
order: if~$|f| \leqslant |g|$ for some function~$g \in X$ then~$f \in X$
and~$\|f\|_X \leqslant \|g\|_X$.

We say that~$X$ has the Fatou property if for any~$f_j \in X$,
$\|f_j\|_X \leqslant 1$ such that~$f_j \to f$ almost everywhere we also have~$f
\in X$ and~$\|f\|_X \leqslant 1$.
The
order dual~$X'$ can be defined as a lattice with the norm~$\|g\|_{X'} = \sup_{\|f\|_X \leqslant 1} \int |f
g|$.  For normed lattices, the Fatou property is equivalent to the order reflexivity~$(X')' = X$.
For example, $\lclass {p}' = \lclass {p'}$ for
all~$1 \leqslant p \leqslant \infty$.

For a quasi-normed lattice~$X$ of measurable functions,
the power~$X^\delta$, $\delta > 0$, is defined by the quasi-norm~$\|f\|_{X^\delta} =
\left\| |f|^{1 \slash \delta}\right\|_X^\delta$.  For instance, $\lclass {p}^\delta = \lclass {p/\delta}$.
We say that~$X$ is~$\delta$-convex if~$X^\delta$ has an equivalent norm.

For a measurable function~$w$ such that~$w > 0$ almost everywhere,
we define weighted lattices~$X (w) = \{f w \mid f \in X\}$
with the quasi-norm~$\|g\|_{X (w)} = \|g w^{-1}\|_X$.
Notice that with this definition the usual
weighted Lebesgue space with quasi-norm~$\left(\int |g|^p w \right)^{1/p}$ is denoted by~$\lclass {p} [w^{-1/p}]$.

For functions~$f$ on~$\mathbb D$ we write~$f_r (z) = f (r z)$,
$0 \leqslant r < 1$.  The radial maximal function is
$$
f^+ (\zeta) = \sup_{0 \leqslant r < 1} |f (r \zeta)|, \quad \zeta \in \mathbb T.
$$
For analytic functions~$f$ the square function is
$$
S f (\zeta) = \left(\int_0^1 |f' (r \zeta)|^2 (1 - r) \mathrm dr\right)^{1/2}, \quad\zeta \in \mathbb T.
$$

Let~$\nplus$ be the set of boundary values of the Smirnov class of analytic
functions on the disc~$\mathbb D$ (see, e.g., \cite {gavrilovetal}, \cite
{privaloven}, \cite {hoffman}).
A Hardy-type space~$X_A$ is defined for a
space~$X$ of measurable functions on~$\mathbb T$ by~$X_A = X \cap \nplus$.
Functions~$f \in X_A$ on~$\mathbb T$ are identified with the
analytic functions~$f$ on~$\mathbb D$ such that~$\lim_{r \to 1-} f_r (\zeta) = f (\zeta)$
for almost all~$\zeta \in \mathbb T$.
For example, from the Lebesgue spaces~$\lclass {p}$, $0 < p \leqslant \infty$
we get the usual Hardy spaces $\left(\lclass {p}\right)_A = \hclass {p}$, but
this definition also yields the Hardy-Lorentz spaces $\hclass {p, q}$,
the weighted Hardy spaces $\hclass {p} [w]$,
the variable exponent Hardy spaces $\hclass {p (\cdot)}$ and many others.
The closedness of~$X_A$ in~$X$ under the assumptions that we use is an easy
consequence of the following result. 

For convenience of notation, we often assume the circle~$\mathbb T$
to be parameterized by the angle~$\theta$ and simply write~$f (\theta)$ in place of~$f (e^{i \theta})$.
The Hardy-Littlewood maximal operator
is
$$
M f (\theta) = \sup_{0 < s \leqslant \pi} \frac 1 {2 s} \int_{\theta - s}^{\theta + s} |f|, \quad -\pi < \theta < \pi.
$$

\begin {theorem}
  \label {t:htssfc}
  Let~$X$ be a quasi-Banach lattice of measurable functions on~$\mathbb T$
  satisfying the Fatou property and $\delta$-convex with some~$\delta > 0$.
  Let~$0 \leqslant r_0 < 1$ and~~$f \in \hol (\mathbb D)$.
  If the Hardy-Littlewood maximal operator~$M$ is bounded in~$X^\delta$,
  the following conditions are equivalent:
  \begin {itemize}
  \item [(i)] $f \in X_A$;
    \item [(ii)] $f^+ \in X$;
  \item [(iii)] $\sup_{r_0 \leqslant r < 1} \|f_r\|_X < \infty$.
  \end {itemize}
  If the Hardy-Littlewood maximal operator~$M$ is bounded in $(X^\delta)'$,
  the above and the following are also equivalent:
  \begin {itemize}
    \item [(iv)]$S f \in X$.
  \end {itemize}  
  The corresponding quasi-norms~$\|f\|_X$, $\|f^+\|_X$, $\sup_{r_0 \leqslant r < 1} \|f_r\|_X$ and
  $|f (0)| + \|S f\|_X$ are equivalent under the respective assumptions. 
\end {theorem}

To give a few examples, the assumptions of Theorem~\ref {t:htssfc} are satisfied by the
rearrangement invariant quasi-Banach lattices with positive lower Boyd index
(by the Lorentz--Shimogaki theorem; see~\cite [Theorem~5.17] {BennettSharpley1988InterpolationOfOperators}).
We briefly mention without going into any detail
that the rearrangement invariant case can also be derived from the
results for the classical Hardy spaces
by interpolation (see~\cite [Theorem~5.16] {BennettSharpley1988InterpolationOfOperators}
and~\cite [Theorem~2.4] {Kisliakov1999InterpolationOfHpSpacesSomeRecentDevelopments}).
It is also curious to note that in particular the Lorentz spaces~$\lclass {p, \infty}$, $0 < p < \infty$,
satisfy the assumptions of Theorem~\ref {t:htssfc}, so these assumptions do not imply
any nontrivial concavity of~$X$.

A particularly important example are the weighted Lebesgue spaces~$\lclass {p} [w^{-1/p}]$ with~$0 < p < \infty$ and Muckenhoupt weights~$w \in \apclass {\infty}$
(since~$w \in \apclass {q}$ for large enough~$p < q < \infty$, the assumption is satisfied with~$\delta = p/q$;
for the theory of Muckenhoupt weights see, e.g., \cite [Chapter~5] {Stein1993HarmonicAnalysisRealVariableMethodsOrthogonalityAndOscillatoryIntegrals}).

Another important example
are variable exponent Lebesgue spaces~$\lclass {p (\cdot)}$ with a log-H\"older continuous exponent~$p (\cdot)$
satisfying~$0 < p_- \leqslant p (\cdot) \leqslant p_+ <  \infty$
(see, e.g., \cite {DieningHarjulehtoHastoRuzicka2011LebesgueAndSobolevSpacesWithVariableExponents};
the analytic variable exponent Lebesgue spaces were studied before to an extent
in~\cite {kokilashvilipaatashvili2006}, \cite {chaconchacon2019}, and in a few other papers).
A further natural generalization where the conditions for the boundedness of the maximal operator
are fairly well understood are the Orlicz-Musielak spaces; see, e.g., \cite {HarjulehtoHasto2019OrliczSpacesAndGeneralizedOrliczSpaces}.

The example~$X = \lclass {\infty}$ shows that the assumptions of Theorem~\ref {t:htssfc} are sharp in a sense,
as in this case~$S f \in X$ is equivalent to~$f \in \mathrm{BMOA}$, which can be verified by our argument
in the proof of Theorem~\ref {t:htssfc}.
The assumption that~$M$ is bounded in~$X^\delta$ in the first part of the theorem
is also sharp in the sense that~$(i) \Rightarrow (iii)$ fails
for~$X = \lclass {\infty} [w]$ with~$w \in \apclass {p}$, $p > 1$.  For instance,
with $w = |\theta|^\alpha \in \apclass {1 + \beta}$, $-\pi < \theta < \pi$ for~$0 < \alpha < \beta$
conditions~$(ii)$ and~$(iii)$ imply~$f (r) = 0$, $r_0 \leqslant r < 1$, and hence~$f = 0$.

To prove Theorem~\ref {t:htssfc}, we will use the following (fairly elementary but not altogether trivial)
basic fact about the duality
of the boundedness of the Hardy-Littlewood maximal operator.
\begin {proposition} [{\cite [Theorem~2] {rutsky2015en}}]
  \label {p:hld}
  Suppose that~$X$ is a normed lattice of measurable functions satisfying the Fatou property.
  If~$M$ is bounded in~$X'$, then it is bounded in both~$X^\alpha$ and~$(X^\alpha)'$ for all~$0 < \alpha < 1$.
\end {proposition}
In particular, by Proposition~\ref {p:hld} the second assumption
in Theorem~\ref {t:htssfc} that~$M$ is bounded in~$(X^\delta)'$ implies
that~$M$ is bounded in both~$X^\delta$ and~$(X^\delta)'$ with smaller values of~$\delta$,
and the first assumption of the theorem that~$M$ is bounded in~$X^\delta = ((X^\delta)')'$ implies
that~$M$ is bounded in~$(X^\delta)'^\alpha$ for~$0 < \alpha < 1$, i.e. that lattice~$(X^\delta)'$ in place of~$X$ also satisfies
the first assumption.

\begin {proof} [Proof of Theorem~\ref {t:htssfc}]
  Suppose that~$M$ is bounded in~$X^\delta$, and hence also in~$(X^\delta)'^\delta$ by Proposition~\ref {p:hld}.
  Both~$X$ and~$(X^\delta)'$ satisfy the so-called property~$(*)$:
for any~$f \in X$, $f \neq 0$ there exists a majorant~$w \geqslant |f|$
such that~$\log w \in \lclass {1}$ and~$\|w\|_X \lesssim \|f\|_X$ with a constant independent of~$f$
(see the remark after~\cite [Proposition~3.1] {Kisliakov1999InterpolationOfHpSpacesSomeRecentDevelopments};
it is easy to see that we can take~$w =
(M |f|^\delta)^{1/\delta} \in \lclass {1, \infty}^{1/\delta} = \lclass {\delta, \infty} \subset \lclass {q}$, $q < \delta$).
Such majorants are moduli~$w = |F|$ of outer functions~$F = \exp (\log w + i H \log w)$,
where~$H$ is the Hilbert transform.


$(i)\Rightarrow(ii)$
Suppose that~$f \in X_A$.
Let~$w$ be the majorant for~$f$ from the property~$(*)$ and~$F$ be the corresponding outer function~$|F| = w$
almost everywhere on~$\mathbb T$.
Then
$(f^+)^\delta \leqslant (F^+)^\delta = (|F|^\delta)^+ \lesssim M (w^\delta)$,
so by the assumption
$$
\|f^+\|_X^\delta = \| (f^+)^\delta \|_{X^\delta} \lesssim
\| M (w^\delta) \|_{X^\delta} \lesssim \| w^\delta \|_{X^\delta} \lesssim \|f\|_X^\delta.
$$

$(ii)\Rightarrow(iii)$
Condition~$f^+ \in X$ trivially implies~$f_r \in X$ uniformly in~$0 \leqslant r < 1$.

$(iii) \Rightarrow (i)$
Suppose that~$f_r \in X$ uniformly in~$r_0 \leqslant r < 1$.
Let~$h \in (X^\delta)'$ with norm~$1$, $w \in (X^\delta)'$ be the majorant~$w \geqslant |h|$
satisfying~$\log w \in \lclass {1}$ from property~$(*)$,
and let~$W$ be the corresponding outer function~$|W| = w$.
By 
~$(i) \Rightarrow (iii)$ for~$(X^\delta)'$ in place of~$X$ (see the remark after Proposition~\ref {p:hld}),
$W_r \in (X^\delta)'$ uniformly in~$0 \leqslant r < 1$.
Thus,
$$
\int |f_r|^\delta |W_r| \leqslant \| |f_r|^\delta \|_{X^\delta} \|W_r\|_{(X^\delta)'}
= \| f_r \|_{X}^\delta \|W_r\|_{(X^\delta)'}
$$
is uniformly bounded in~$r_0 \leqslant r < 1$,
hence also for~$0 \leqslant r < 1$ with the same constant due to the subharmonicity
of~$|G|^\delta$ for the analytic function~$G = f W^{1/\delta}$,
which means that~$G$
belongs to the classical Hardy space~$\hclass {\delta}$.
The nontangential boundary values of~$G$ define the same for~$f = G W^{-1/\delta}$,
and these satisfy
$$
\int |f|^\delta |h| \leqslant \int |f|^\delta w = \int |G|^\delta < \infty
$$
uniformly in~$h$.
It follows by the Fatou property
that~$|f|^\delta \in X^\delta$ and~$f \in X$.

We now turn to the second part of Theorem~\ref {t:htssfc} concerning condition~$(iv)$.
We note that the following argument is an example of
a general theory of square functions and can be naturally carried out in many settings;
see, e.g., \cite [Chapter I, \S6.4 and~\S8.23] {Stein1993HarmonicAnalysisRealVariableMethodsOrthogonalityAndOscillatoryIntegrals}.

{}
By Proposition~\ref {p:hld}, making~$\delta$ smaller if necessary allows us to assume
that~$0 < \delta \leqslant 1$
and that~$M$ is bounded in both~$X^\delta$ and~$(X^\delta)'$.

Notice that by the boundedness of the maximal operator
$$
\| |f (0)|^\delta\|_{X^\delta} \leqslant \|M (|f|^\delta)\|_{X^\delta} \lesssim \|f\|_X^\delta,
$$
so we may always assume that~$f (0) = 0$.
By the same boundedness of the maximal operator and the Fatou property, $\lclass {\infty} \subset (X^\delta)'$,
so~$X^\delta \subset \lclass {\infty}' = \lclass {1}$ and~$X \subset \lclass {1}^{1/\delta} = \lclass {\delta}$.

We will first establish the square function characterization for
weighted Lebesgue spaces~$X = \lclass {p} [w^{-1/p}]$,
$0 < p < \infty$,
uniformly in the Muckenhoupt weights~$w \in \apclass {1}$.
Let
$$
H = \lclass {2} [(0, 1), \left[r \log 1/r\right]^{-1/2}] \otimes \mathbb C^2.
$$
Let~$\mathcal S$ be a singular integral operator
$\mathcal S f (\theta) = \int_{-\pi}^\pi K (\theta - \eta) f (\eta) \mathrm d\eta$
in the sense of \cite [Chapter~I, \S1.5] {Stein1993HarmonicAnalysisRealVariableMethodsOrthogonalityAndOscillatoryIntegrals}
with kernel
$$
K (\theta) = \nabla \{r \mapsto P_r (\theta)\}_{0 \leqslant r < 1} \in H, \quad \theta \in (-\pi, \pi) \setminus \{0\},
$$
where~$P_r$ is the Poisson kernel.
It is well known that the scalar theory of Calder\'on-Zygmund operators and real Hardy spaces
is naturally extended to the setting of operators between spaces taking values in separable Hilbert spaces
(see, e.g., \cite [Chapter~I, \S6.4] {Stein1993HarmonicAnalysisRealVariableMethodsOrthogonalityAndOscillatoryIntegrals}).
Notice that for analytic functions~$f$
$$
\mathcal S f (\theta, r) = \nabla (P_r * f) (\theta, r) = \nabla f (r e^{i \theta}),
\quad 0 \leqslant r < 1, \quad -\pi < \theta < \pi,
$$
and so~$\|\mathcal S f (\theta, r)\|_{\mathbb C^2}^2 \leqslant 2 |f' (r e^{i \theta})|^2$.
The square function~$S$ is expressed
up to pointwise equivalence
as~$S f \asymp \|\mathcal S f\|_H$ almost everywhere
due to, e.g.,~$\log 1/r \asymp (1 -r)$ for~$1/4 \leqslant r < 1$ and
$$
\sup_{|z| \leqslant 1/4} |f'|^\delta \lesssim \int_{1/4 \leqslant |z| \leqslant 1/2} |f'|^\delta \lesssim \|S f\|_{\lclass {\delta}}^\delta \lesssim \|S f\|_X^\delta
$$
by the H\"older inequality.

$\mathcal S$ is a Calder\'on-Zygmund operator that satisfies the standard smoothness conditions
$\|(\partial / \partial \theta)^m K\|_H \lesssim 1/|\theta|^{(m + 1)}$,
$\theta \in (-\pi, \pi) \setminus \{0\}$ for all~$m \geqslant 0$ and is an isometry
as~$(\lclass {2})_{\mathbb R} \ominus \mathbb R \to \lclass {2} [H]$ up to a
constant ($\|f\|_{\lclass {2}}^2 = 2 \|\mathcal S f\|_{\lclass {2} [H]}^2$
for real-valued~$f$ with average~$0$; see~\eqref {e:greenf} below).
The adjoint operator~$\mathcal S^*$ taking $H$-valued functions into scalar functions
is also a Calder\'on-Zygmund operator satisfying the same assumptions on the kernel.

The real weighted Hardy space~$\hrclass {p} [w^{-1/p}]$ can be defined as 
 the set of distributions~$\mu$
satisfying~$\mu^+ = \sup_{0 \leqslant r < 1} |P_r * \mu| \in \lclass {p} [w^{-1/p}]$
(see~\cite [Chapter~VI] {StrombergTorchinsky1989WeightedHardySpaces}).
The weighted real Hardy space theory \cite [Chapter~XI, Theorem~12] {StrombergTorchinsky1989WeightedHardySpaces}
shows that both~$\mathcal S$ and $\mathcal S^*$
are bounded as operators between the real Hardy spaces~$\hrclass {p} [w^{-1/p}]$ and~$\hrclass {p} [H, w^{-1/p}]$
with the norm depending only on the~$\apclass {1}$-constant of the weight~$w$
for all~$0 < p < \infty$.  In fact, the weighted real Hardy space theory
in~\cite {StrombergTorchinsky1989WeightedHardySpaces} is developed
under a much weaker assumption~$w \in \apclass {\infty}$,
but it can be easily recovered from~$w \in \apclass {1}$  by Theorem~\ref {t:extr} below
as long as~$0 < p < \infty$ is arbitrary.
The weighted results with~$w \in \apclass {1}$
can also be achieved by a natural generalization of the standard theory of the real Hardy spaces (see, e.g., \cite [Chapter~III] {Stein1993HarmonicAnalysisRealVariableMethodsOrthogonalityAndOscillatoryIntegrals}),
but it is difficult to find a suitable reference.

It is well known that~$\hclass {p} [w^{-1/p}] \subset \hrclass {p} [w^{-1/p}]$;
see, e.g., \cite [Chapter~III, \S4.1, Proposition~1] {Stein1993HarmonicAnalysisRealVariableMethodsOrthogonalityAndOscillatoryIntegrals}).
The fact that a function~$f = \sum_{j \geqslant 0} c_j z^j \in \hclass {\delta}$
has boundary values~$f|_{\mathbb T} = \mu$, $f_r = P_r * \mu$ in the sense of distributions, can easily be derived from
a well-known estimate~$|c_j| \lesssim j^{1/\delta - 1}$, or any polynomial growth estimates on~$c_j$.

Thus, 
operators~$\mathcal S$ and~$\mathcal S^*$
are also bounded between the corresponding weighted analytic Hardy spaces.
This shows that
$$\|S f\|_{\lclass {p} [w^{-1/p}]}
\asymp \|\mathcal S f\|_{\lclass {p} [H, w^{-1/p}]} \lesssim \|f\|_{\hrclass {p} [w^{-1/p}]} \lesssim \|f\|_{\hclass {p} [w^{-1/p}]},
$$
i.e.~$(i) \Rightarrow (iv)$ for~$X = \lclass {p} [w^{-1/p}]$ with~$w \in \apclass {1}$.

To obtain the converse implication~$(iv) \Rightarrow (i)$ for these~$X$,
we employ the Green formula (see, e.g., \cite [Section X.D] {Koosis1998IntroductionToHpSpaces})
  \begin {equation}
    \label {e:greenf}
    \int_{-\pi}^\pi f g\, \mathrm d\theta =
    2 \int_{\mathbb D} (\nabla f \cdot \nabla g) \log (1/|z|)\, \mathrm dx\, \mathrm dy =
    2 \int_{-\pi}^\pi (\mathcal S f, \mathcal S g)_H\, \mathrm d\theta
  \end {equation}
  valid for continuous real-valued harmonic~$f$ and~$g$ such that~$f (0) = 0$.
  Taking the harmonic extensions of a given continuous function~$f$ on~$\mathbb T$ with average 0 and
  an arbitrary continuous function~$g$ on~$\mathbb T$ yields the reproducing formula
  $f = 2 \mathcal S^* \mathcal S f$ almost everywhere that naturally extends to complex-valued
  functions~$f$.
  In particular,
  it follows that~$\|f_r\|_X \lesssim \|\mathcal S^*\|_{X_A (H) \to X_A} \|S f_r\|_{X}$.
  By the monotonicity of the weight in the definition of the square function,
  $S f_r \leqslant (1/r) S f$,
  so by~$(i) \Rightarrow (iii)$
  $$
  \|f\|_{X_A} \lesssim \sup_{1/2 \leqslant r < 1} \|f_r\|_X \lesssim \|S f\|_X.
  $$

  Thus, $(i) \Leftrightarrow (iv)$ is true for~$X = \lclass {p} [w^{-1/p}]$ with~$w \in \apclass {1}$.
  The same equivalence
  for arbitrary lattices~$X$ follows at once from a suitable extrapolation theorem.
  Namely, it suffices to apply the following result
  to sets
  $$
  \mathcal F = \left\{ (S f, f) \mid f \in \hclass {\delta} \right\}, \quad
  \mathcal F' = \left\{ (f, S f) \mid f \in \hclass {\delta} \right\}.
  $$
\end {proof}
  \begin {theorem} [{\cite [Theorem~4.6] {weightsandextrapolation}}]
    \label {t:extr}
    Let~$\mathcal F$ be a set of couples
  of measurable functions on a space of homogeneous type~$\mathcal X$.
      Let~$X$ be a quasi-Banach lattice of measurable functions on~$\mathcal X$
  satisfying the Fatou property and $\delta$-convex with some~$\delta > 0$.
  Suppose that the Hardy-Littlewood maximal operator~$M$ is bounded in $(X^\delta)'$,
  and $\int |f|^\delta w \lesssim \int |g|^\delta w$
  uniformly in~$(f, g) \in \mathcal F$ and~$w \in \apclass {1}$ (i.e. uniformly on the sets of~$\apclass {1}$
  weights~$\{w \mid M w \leqslant C w\}$ with a given arbitrary constant~$C$).
  Then~$\|f\|_X \lesssim \|g\|_X$ uniformly in~$(f, g) \in \mathcal F$.
  \end {theorem}
  In~\cite {weightsandextrapolation}, Theorem~\ref {t:extr}
  was only established for Banach lattices~$X$ with a different set of assumptions.
  For the sake of clarity, we will now give a simple proof of this result.
  \begin {proof} [Proof of Theorem~\ref {t:extr}]
  Let~$(f, g) \in \mathcal F$.  By the Fatou property~$X^\delta = ((X^\delta)')'$,
  so there exists some~$h \in (X^\delta)'$ with norm~$1$ such that
  $$
  \|f\|_X^\delta = \| |f|^\delta \|_{X^\delta} \leqslant 2 \int |f|^\delta |h|.
  $$
  The well-known Rubio de Francia construction
  $$
  w_0 = |h|, \quad
  w_{j + 1} = M w_j / (2 \|M\|_{(X^\delta)' \to (X^\delta)'}), \quad
  w = \sum_{j \geqslant 0} w_j
  $$
  yields a majorant~$w \geqslant |h|$, $M w \lesssim w$ uniformly in~$h$,
  such that~$\|w\|_{(X^\delta)'} \leqslant 2$.
  It follows that
  $$
  \|f\|_X^\delta \leqslant 2 \int |f|^\delta w \lesssim \int |g|^\delta w \leqslant
  \| |g|^\delta \|_{X^\delta} \|w\|_{(X^\delta)'} \leqslant 2 \|g\|_X^\delta.
  $$
  \end {proof}

  \section {Composition operators on weighted Bloch spaces}
  \label {s:comp}
Consider a weight function $\omega$, that is, a non-decreasing, continuous,
unbounded function $\omega : [0, 1) \to (0,+\infty)$. We extend $\omega$ to a radial weight on $\Dbb$
setting $\omega(z) = \omega(|z|)$, $z \in \Dbb$.

Given a radial weight $\omega$, the Bloch spaces $\bloch^\omega (\Dbb)$ consists of functions
$f \in \hol(\Dbb)$ such that
\begin{equation}\label{e_bloch}
|f^\prime (z)| \le C \omega(z),\quad z \in \Dbb,
\end{equation}
for some constant $C > 0$.
If $\omega_1$ and $\omega_2$ are equivalent radial weights, then the identities
\[
\| f\|_{\bloch^{\omega_j} (\Dbb)} = |f(0)| + \sup_{z\in\Dbb}
\frac{|f^\prime (z)|}{\omega_j(z)},\quad j = 1, 2,
\]
define equivalent norms on the Banach space
$\bloch^{\omega_1} (\Dbb) = \bloch^{\omega_2} (\Dbb)$.

By definition, $\omega$ is said to be
log-convex
if $\log \omega (t)$ is a convex function of~$\log t$,
$0 < t < 1$.
  
Let $\ph: \Dbb \to \Dbb$ be a holomorphic function.
The composition operator $C_\varphi$ is defined as
\[
C_\ph f(z) = f(\ph(z)),\quad f\in\hol(\Dbb).
\]

The following lemma is a straightforward corollary of Theorem~1.2 from \cite{AbD15}.

\begin{lemma}\label{l_BLMS}
Let $\omega$ be a log-convex weight function.
There exist $f_1, f_2 \in \bloch^\omega$ such that
\[
|f^\prime_1(z)|^2 + |f^\prime_2(z)|^2 \ge \omega^2(z), \quad z\in\Dbb.
\]
\end{lemma}

A quasi-normed lattice~$X$ is said to have order continuous quasi-norm if
for any nonincreasing sequence~$f_j \in X$, $f_j \geqslant 0$,
$f_j \to 0$ almost everywhere, it follows that~$\|f_j\|_X \to 0$.
It is easy to see that order continuity is equivalent to the dominated convergence theorem:
for any~$f_j \to f$ almost everywhere such that~$|f_j| \leqslant g \in X$ it follows that~$\|f_j - f\|_X \to 0$.

The following theorem was obtained for the classical Hardy spaces~$\hclass {p}$
and the standard Bloch space in~\cite {Kw1}.
See also~\cite {Kw2}, \cite {ChH1}, \cite {ChH2} for multidimensional generalizations.
\begin{theorem}
  \label {t:cph} 
Let $\ph: \Dbb\to \Dbb$ be a holomorphic function.
Let $\omega$ be a log-convex weight function. 
Suppose that~$X$ is a quasi-normed lattice of measurable functions on~$\mathbb T$ satisfying the assumptions
of Theorem~\ref {t:htssfc}.
Then the following properties are equivalent:
\begin{itemize}
  \item [(i)] $\cph: \bloch^\omega \to X_A$  is a bounded operator;
  \item [(ii)] $\left(\int_0^1 \omega^2 (\ph(r\za)) |\ph^\prime(r\za)|^2 (1-r)\, dr\right)^{\frac12} \in X$.
\end{itemize}
If, additionally, $X$ has order continuous quasi-norm, then the above implies that
$\cph: \bloch^\omega \to X_A$  is a compact operator.
\end{theorem}
\begin{proof}
  $(i)\Rightarrow(ii)$
  Notice that~$(C_\ph f)' = (C_\ph f') \ph'$.
Let $f_1, f_2\in \bloch^\omega$ be those provided by Lemma~\ref{l_BLMS}.
By $(i)$ and Theorem~\ref{t:htssfc},
\[
\sum_{j=1}^2 \left(\int_0^1 |f_j^\prime(\ph(r\za))|^2 |\ph^\prime(r\za)|^2 (1-r)\, dr\right)^{\frac12}\in X.
\]
Since $X$ is a lattice, we obtain
\begin {equation*}
  \label {e:az}
A (\zeta) = \left(\int_0^1 \omega^2(\ph(r\za)) |\ph^\prime(r\za)|^2 (1-r)\, dr\right)^{\frac12}\in X,
\end {equation*}
as required.

$(ii)\Rightarrow(i)$
Let $f\in\bloch^\omega$. 
By the definition of $\bloch^\omega$,
$$
|f^\prime(\ph(r\za))|^2 \leqslant \|f\|_{\bloch^\omega}^2 \omega^2(\ph(r\za)),
\quad 0<r<1, \quad \za\in\Tbb.
$$
Therefore,
\[
S (C_\ph f) = \left(\int_0^1 |(\cph f')(r\za)|^2 |\ph^\prime(r\za)|^2 (1-r)\, dr\right)^{\frac{1}{2}}
\lesssim A (\zeta) \|f\|_{\bloch^\omega}.
\]
For any~$z \in \mathbb D$, $f (z) = f (0) + \int_0^z f' (\zeta)\,\mathrm d\zeta$, so
$$
|f (z)| \leqslant |f (0)| + \int_0^z |f' (\zeta)|\, \mathrm d|\zeta| \leqslant
(1 + |z| \omega (|z|)) \|f\|_{\bloch^\omega}.
$$
By Theorem~\ref{t:htssfc},
$$
\|\cph f\|_{X_A} \lesssim |f (\ph (0))| + \|S (C_\ph f)\|_X \lesssim
(1 + |\ph (0)| \omega (|\ph (0)|) + \|A\|_X)  \|f\|_{\bloch^\omega},
$$
so~$(i)$ holds.

Now suppose that~$(ii)$ is satisfied and~$X$ has order continuous quasi-norm.
Let~$g_j$ be a bounded sequence in~$\bloch^\omega$.
It is uniformly bounded on compact sets in~$\mathbb D$, so there exists some analytic function~$g$
such that a subsequence~$g_{j'}$ of~$g_j$ converges to~$g$ uniformly on compact sets in~$\mathbb D$.
Clearly, $g \in \bloch^\omega$, and it suffices to show that~$C_\ph g_{j'} \to C_\ph g$ in~$X_A$.

Let~$h_j = g_{j'} - g$ and suppose that~$\varepsilon > 0$.
Notice that
$$
I_2 = \left(\int_t^1 \omega^2(\ph(r\za)) |\ph^\prime(r\za)|^2 (1-r)\, dr\right)^{\frac12} \to 0, \quad
t \to 1-,
$$
so by the order continuity of the quasi-norm~$\|I_2\|_X \leqslant \varepsilon$
with some~$0 < t < 1$.
Since~$\ph (t \bar{\mathbb D})$ is compact,
$$
I_1 = \left(\int_0^t |h_j' (\ph (r \zeta))|^2 |\ph^\prime(r\za)|^2 (1-r)\, dr\right)^{\frac12} \to 0,
\quad
j \to \infty,
$$
and~$I_1 \leqslant A$, so by the dominated convergence
$\|I_1\|_X \to 0$ as~$j \to \infty$.
Thus,
$\|S (C_\ph h_j)\|_X \lesssim \|I_1\|_X + \|I_2\|_X \lesssim \varepsilon$
for large enough~$j$.
By the arbitrariness of~$\varepsilon$ it follows that~$C_\ph h_j \to 0$ in~$X_A$ as claimed.
\end{proof}

The following simple example shows that without the order continuity assumption~$C_\varphi$
may fail to be compact.
Let~$\gamma > 1$, $0 < \alpha < 1$, $\omega (r) = (1 - r)^{-\gamma}$, $\varphi (z) = 1 - (1 - z)^\alpha$,
$b_j \to 1+$, $g_j = (b_j - z)^{1 - \gamma}$, $g = (1 - z)^{1 - \gamma} \in \bloch^\omega$.
Then~$|g_j'| \leqslant |g'|$, and~$g_j \to g$ uniformly on compact sets.
Also, $C_\ph g = (1 - z)^{\alpha (1 - \gamma)}$,
$|g_j| \leqslant |g|$, and hence~$|C_\ph g_j| \leqslant |C_\ph g| \in \lclass {p, \infty}$
with~$1/p = \alpha (\gamma - 1)$.

Suppose that~$|1 - \zeta| \leqslant 1/2$ and let~$t = 1 - |1 - \zeta|$.
Notice that for~$0 < r < t$,
$1 - |1 - (1 - r \zeta)^\alpha| \asymp (1 - r)^\alpha$, so
$$
I_1^2 = \int_0^t \omega^2 (\ph (r \zeta)) |\ph' (r \zeta)|^2 (1 - r) \mathrm dr
\lesssim
\int_0^t (1 - r)^{2 \alpha (1 - \gamma) - 1} \mathrm dr
\lesssim
|1 - \zeta|^{2 \alpha (1 - \gamma)}.
$$
For~$t \leqslant r < 1$,
$1 - |1 - (1 - r \zeta)^\alpha| \asymp |1 - \zeta|^\alpha$, and
$$
I_2^2 = \int_t^1 \omega^2 (\ph (r \zeta)) |\ph' (r \zeta)|^2 (1 - r) \mathrm dr
\lesssim
|1 - \zeta|^{2 \alpha (1 - \gamma) - 2}
\int_t^1 (1 - r) \mathrm dr
\lesssim
|1 - \zeta|^{2 \alpha (1 - \gamma)}.
$$
The last estimate is also true for~$|1 - \zeta| > 1/2$ with~$t = 0$.
Thus, $A (\zeta) \lesssim I_1 + I_2 \lesssim |1 - \zeta|^{\alpha (1 - \gamma)} \in \lclass {p, \infty}$,
so condition~$(ii)$ of Theorem~\ref {t:cph} is satisfied
with the Lorentz space~$X = \lclass {p, \infty}$.
However, it is easy to see that for any bounded function~$h$, and in particular for~$h = C_\ph g_j$,
$$
\|C_\ph g - h\|_X^{p} \geqslant
\lim_{t \to \infty} t |\{ |(1 - \zeta)^{\alpha (1 - \gamma)} - h|^{p} > t\}|
\geqslant
\frac 1 2 \lim_{t \to \infty} t |\{ |1 - \zeta|^{-1} > t\}| = \frac 1 2.
$$
  {}



\providecommand{\MR}{\relax\ifhmode\unskip\space\fi MR }
\providecommand{\MRhref}[2]{%
  \href{http://www.ams.org/mathscinet-getitem?mr=#1}{#2}
}
\providecommand{\href}[2]{#2}

\end{document}